\documentclass{amsart}
\usepackage[parfill]{parskip}

\usepackage{graphicx}
\usepackage{ifpdf}

\theoremstyle{plain}
\newtheorem{thm}{Theorem}[section]
\newtheorem{lemma}[thm]{Lemma}
\newtheorem{prop}[thm]{Proposition}
\newtheorem{cor}[thm]{Corollary}

\newtheorem*{mainthm}{Theorem \ref{main-thm}}
\newtheorem*{maincor1}{Corollary \ref{main-cor1}}

\theoremstyle{definition}
\newtheorem{defn}[thm]{Definition}
\newtheorem{ex}[thm]{Example}

\theoremstyle{remark}
\newtheorem{rmk}[thm]{Remark}

\newcommand{\incgraphics}[1]{ %
   \ifpdf %
      \includegraphics{#1.pdf} %
   \else %
      \includegraphics{#1.eps} %
   \fi %
   }

\newcommand{\R}{\mathbb{R}}

\newcommand{\Z}{\mathbb{Z}}
\newcommand{\U}{\mathcal{U}}
\newcommand{\Xhat}{\hat X}
\newcommand{\h}{\mathfrak{h}}

\newcommand{\Hhat}{{\hat H}}

\newcommand{\Hop}{H^{\mathrm{op}}}
\newcommand{\Hlt}{H^{<}}
\newcommand{\pop}{p^{\mathrm{op}}}
\newcommand{\plt}{p^{<}}

\newcommand{\normi}[1]{\left\|#1\right\|_1}
\newcommand{\comment}[1]{\relax}
\DeclareMathOperator{\rnk}{-rank}
\newcommand{\drank}{d\rnk}

\renewcommand{\i}{\mathfrak{i}}
\renewcommand{\r}{\mathfrak{r}}

\DeclareMathOperator{\LSup}{\Lambda-Supp}
\DeclareMathOperator{\HopSup}{\Hop-Supp}

\DeclareMathAlphabet{\mathitbf}{OML}{cmm}{b}{it}
\newcommand{\C}{\mathitbf{C}}

\title{Finite asymptotic dimension for CAT(0) cube complexes}

\author{Nick Wright}

\keywords{Asymptotic dimension; CAT(0) cube complexes; small cancellation groups}
\subjclass[2000]{54F45; 20F69; 20F65}

\begin{document}

\begin{abstract}
In this paper we prove that the asymptotic dimension of a finite-dimensional CAT(0) cube complex is bounded above by the dimension. To achieve this we prove a controlled colouring theorem for the complex. We also show that every CAT(0) cube complex is a contractive retraction of an infinite dimensional cube. As an example of the dimension theorem we obtain bounds on the asymptotic dimension of small cancellation groups.
\end{abstract}

\maketitle

\section*{Introduction}

CAT(0) cube complexes are a higher dimensional analogue of trees. They are a class of non-positively curved spaces with useful combinatorial and geometric properties. Many interesting classes of groups, for example small cancellation groups, act properly on CAT(0) cube complexes, allowing geometric properties of the groups to be deduced from the corresponding properties of the cube complexes. Indeed by the results of \cite{OW}, `generic' groups groups act on finite dimensional CAT(0) cube complexes; more precisely random groups at density less than $\frac 16$ act freely and cocompactly on finite dimensional CAT(0) cube complexes with probablility tending exponentially to 1.

It is well known that 1-dimensional CAT(0) cube complexes, i.e.\ trees, have asymptotic dimension at most 1. A product of trees gives a higher dimensional CAT(0) cube complex, and by the Hurewicz Theorem for asymptotic dimension \cite{BD}, one again has that the asymptotic dimension is bounded above by the dimension of the cube complex. In this paper, we show that this holds in full generality. Note that we do not require the cube complex to be locally finite.

This paper extends the result of \cite{BCGNW}, where it is shown that finite dimensional CAT(0) cube complexes have property A. It also extends the result of \cite{Dcox} where it is shown that the asymptotic dimension of a right-angled Coxeter group does not exceed the dimension of its Davis complex.

The study of asymptotic dimension for cube complexes is motivated in part by the example of Thompson's group $F$. The question of whether Thompson's group is amenable is one that has attracted a lot of attention (see \cite{AA,Sh,B}), however at present the arguments for and against amenability are incomplete. As Sapir remarked in a private communication, the related question of whether Thompson's group is exact ``is considered almost as hard as amenability,'' and the results of this paper are connected with this problem. In \cite{Dopen} Dranishnikov asked what the dimension growth of Thompson's group is. Establishing polynomial dimension growth for $F$ would imply exactness, by \cite{D}. Thompson's group acts on an infinite dimensional CAT(0) cube complex (see \cite{F}) with the property that the dimension of the cubes increases polynomially with the distance from an orbit. In establishing bounds on the asymptotic dimension of CAT(0) cube complexes we provide a first step in proving the exactness of Thompson's group.

We prove the following result.

\begin{mainthm}
Let $X$ be a CAT(0) cube complex of dimension $D$. Then for all $\varepsilon>0$ there exists an $\varepsilon$-Lipschitz cobornologous map from $X$ to a CAT(0) cube complex of dimension at most $D$. Thus $X$ has asymptotic dimension at most $D$.
\end{mainthm}

While one certainly can have equality of the dimension and asymptotic dimension, one could not expect to have this in general: the dimension of the cube complex can be increased without changing the asymptotic dimension by taking a product with a finite cube complex.

As an immediate corollary of the theorem we have:

\begin{maincor1}
Let $G$ be a group admitting a proper isometric action on a CAT(0) cube complex of dimension $D$. Then $G$ has asymptotic dimension at most $D$.
\end{maincor1}

As an example of this, we can apply the results of Wise \cite{W} to obtain upper bounds for the asymptotic dimension of $B(4)-T(4)$ and $B(6)$ small cancellation groups in terms of the presentation complex for the group, see Example \ref{small-cancellation}. These bounds are local, and can be determined from a presentation of the group. If $G$ is a $B(4)-T(4)$ small cancellation group then the asymptotic dimension of $G$ is at most $c/2$ where $c$ is the maximal circumference of a cell in the presentation complex of $G$. In other words $c$ is the maximal length of a relator (or at most twice the maximal length if some relators are of odd length). If $G$ is a $B(6)$ group then the asymptotic dimension of $G$ is at most $\max\{c,l\}$, where $c$ is as above and $l$ is the maximal cardinality of a complete graph in the generalised link of a vertex.

\medskip

To prove Theorem \ref{main-thm}, the strategy is to construct a Lipschitz map from a CAT(0) cube complex $X$ into a quotient complex $\Xhat$, with Lipschitz constant strictly less than one. By iterating the process we obtain arbitrarily contractive maps into a cube complex of dimension no greater than $X$, thus obtaining a bound on the asymptotic dimension.

In section \ref{sec:colouring} we introduce the concept of controlled colourings on the set of hyperplanes; these are used in section \ref{sec:interp} to produce the quotient complex $\Xhat$. The main result of section \ref{sec:colouring} is an existence theorem for controlled colourings. In section \ref{sec:projection} we prove a projection theorem. Specifically we prove that every CAT(0) cube complex is a contractive retraction of an infinite dimensional cube. The retraction is constructed as an infinite composition of functions, and we give hypotheses under which a general infinite composition is well-defined. In section \ref{sec:interp} we use the controlled colouring to construct a large-scale contractive map from a CAT(0) cube complex $X$ into a quotient $\Xhat$. This map is 1-Lipschitz on small scales, and to complete the proof we carry out an interpolation argument, using the  projection theorem from section \ref{sec:projection}, to produce a map with Lipschitz constant strictly less than 1. The asymptotic dimension bound then follows.

\section{Preliminaries}
\label{prelim}

In this section we introduce some notation and basic structure on CAT(0) cube complexes. For further information on CAT(0) cube complexes see \cite{BH,CN,G-hyp-grps,BN,R,S}.

Recall that a geodesic metric space $(X,d)$ is CAT(0) if all geodesic triangles are slimmer than the corresponding Euclidean triangle. Consider a cell complexes built out of Euclidean cubes $[0,1]^n$, and with isometric attaching maps that take faces to faces of the same dimension. The metric on the cells extends to a path metric on the cell complex. If the resulting metric is CAT(0) then the complex is said to be a CAT(0) cube complex. The CAT(0) metric condition for a cube complex is equivalent to a combinatorial condition on the cells \cite{G-hyp-grps}: $X$ is a CAT(0) cube complex if and only if it is simply connected and the link of each vertex is a flag complex.

CAT(0) cube complexes can also be equipped with a combinatorial metric. For $x,y$ vertices of $X$, let $d(x,y)$ denote the minimum number of edges required to connect $x$ and $y$. This is called the \emph{edge-path metric}. We remark that this corresponds to the path metric on $X$ produced by equipping each cube with the $l^1$ metric instead of the Euclidean ($l^2$) metric.

Recall that a CAT(0) cube complex can be equipped with a set of \emph{hyperplanes}. Each hyperplane divides the vertex set of $X$ into two \emph{halfspaces} (or \emph{sides}) \cite{NR,S}. We denote the set of hyperplanes of $X$ by $H$. Given two hyperplanes $h,k$, one obtains four possible intersections of halfspaces, and $h,k$ are said to \emph{intersect} if each `quadrant' is non-empty. This occurs if and only if $h$ and $k$ cross a common cube, moreover (cf. \cite{S}), given a maximal collection of pairwise intersecting hyperplanes there is a unique cube which all of them cross. The dimension of a CAT(0) cube complex $X$ is thus the maximum number of pairwise intersecting hyperplanes.

For $x,y$ vertices of $X$, the \emph{interval} from $x$ to $y$, denoted $[x,y]$, is the subcomplex whose vertices lie in all halfspaces containing both $x$ and $y$. Given three vertices $x,y,z$ the intersection $[x,y]\cap[y,z]\cap[z,x]$ contains a unique vertex (cf.\ \cite{R}) called the \emph{median} of $x,y,z$.

If $x$ is a vertex of $X$, then $x$ determines a choice of orientation for each hyperplane, that is a selection of one halfspace for each hyperplane. Let $\h_x$ denote the set of halfspaces containing $x$. These halfspaces have non-empty intersection, indeed their intersection is $x$. The halfspaces are, in particular pairwise intersecting. Moreover, if we fix a basepoint $x_0$ in $X$, then a choice of orientation (i.e.\ a halfspace) for each hyperplane of $X$ will determine a vertex $x\in X$ if and only if the halfspaces intersect pairwise and only finitely many\footnote{If infinitely many choices differ from $\h_{x_0}$ then one obtains a point in the combinatorial boundary.} differ from $\h_{x_0}$.

Given any subset $\Hhat$ of the set of hyperplanes of $X$, one can construct a CAT(0) cube complex $\Xhat$. This construction, which is a generalisation of Sageev's construction \cite{S}, is discussed in detail in \cite{CN,BN}. A vertex of $\Xhat$ is defined to be a set of halfspaces, one corresponding to each hyperplane in $\Hhat$, whose total intersection is non-empty. When $\Hhat=H$ one recovers the original complex $X$ as the sets of halfspaces each yield a single vertex of $X$, while in general the space $\Xhat$ is a quotient of $X$. The set of hyperplanes of $\Xhat$ is (or rather is canonically identified with) $\Hhat$. On vertices the quotient map restricts the family of orientations $\h_x$ to a family of orientations of the hyperplanes in $\Hhat$. This is extended affinely to cubes.

\medskip
We now fix a basepoint $x_0$ in $X$. This will give rise to an ordering on $H$, and we will establish some basic properties.

For a hyperplane $h$, we denote the halfspace containing $x_0$ by $h^-$ and denote the halfspace not containing $x_0$ by $h^+$. We refer to these as respectively the inward and outward halfspaces of $h$.
\begin{defn}
For $h,k\in H$, we say that $k$ \emph{separates} $x_0$ from $h$, denoted $k<h$, if $k^+$ strictly contains $h^+$ or equivalently $k^-$ is strictly contained in $h^-$.

If $k<h$ and there is no $j$ in $H$ such that $k<j<h$ then the hyperplane $k$ is said to be a \emph{predecessor} of $h$.

For $h,k\in H$ we say that $h$ is \emph{opposite} $k$ if $h^+,k^+$ are disjoint.
\end{defn}

Note that a hyperplane $k$ separates $x_0$ from $h$ if and only if it separates $x_0$ from all vertices adjacent to $h$.

For any two hyperplanes $h\neq k$, \emph{exactly one} of the following four conditions holds: $h$ intersects $k$, $h<k$, $k<h$ or $h$ opposite $k$. The conditions $h<k$, $k<h$ and $h$ opposite $k$ correspond to the cases where $h^-\cap k^+$, $h^+\cap k^-$, $h^+\cap k^+$ respectively are empty, while $h$ intersects $k$ is the case where none of these are empty.

\begin{lemma}
For $h\in H$, the predecessors of $h$ all meet. In particular $h$ can have at most $d$ predecessors where $d$ is the dimension of $X$.
\end{lemma}

\begin{proof}
We first observe that two predecessors of a hyperplane $h$ cannot be opposite. Suppose $k_1$ is a predecessor of $h$ and consider a hyperplane $k_2$ which is opposite $k_1$. Then $k_1^+$ contains $h^+$ but is disjoint from $k_2^+$. Thus $h^+$ must also be disjoint from $k_2^+$ i.e.\ $k_2$ is also opposite $h$, so it is not a predecessor of $h$. Now suppose that $k_1,k_2$ are predecessors of $h$. Then we cannot have $k_1<k_2<h$ or $k_2<k_1<h$ as this would contradict the fact that $k_1$ (resp.\ $k_2$) is a predecessor of $h$. Thus if $k_1,k_2$ are predecessors of $h$ then $k_1$ must intersect $k_2$.
\end{proof}

We conclude this section with the following proposition, which gives an alternative characterisation of predecessors.

\begin{prop}\label{H_x}
For $x$ a vertex of $X$, let $H_x$ denote the set of hyperplanes which are adjacent to $x$ and separate $x$ from $x_0$.
\begin{enumerate}
\item Let $h\in H$ and let $x$ be a vertex in $h^+$. Then $x$ is of minimal distance from $x_0$ if and only if $H_x=\{h\}$.

\item Let $h\in H$, and let $y$ be the vertex adjacent to $h$ of minimal distance from $x_0$. That is $y$ is adjacent across $h$ to the vertex $x\in h^+$ of minimal distance from $x_0$. Then $k$ is a predecessor of $h$ if and only if $k\in H_y$.
\end{enumerate}
\end{prop}

This proposition gives a connection between predecessors and normal cube paths (cf. \cite{NR}). The set $H_x$ is the set of hyperplanes meeting the first cube on the normal cube path from $x$ to $x_0$. The proposition says that the set of predecessors of $h$ is the set of hyperplanes meeting the first cube on the normal cube path from $y$ to $x_0$, where $y$ is the vertex adjacent to $h$ of minimal distance to $x_0$, 

\begin{proof}
For (1) first suppose that $x$ is of minimal distance from $x_0$. Since each $k$ in $H_x$ separates $x_0,x$, crossing any of these reduces the distance to $x_0$. If there were any $k\neq h$ in $H_x$, then crossing this we would remain in $h^+$. Hence by minimality of $d(x_0,x)$, and since $H_x$ is non-empty, it follows that $H_x=\{h\}$.

Conversely suppose that $H_x=\{h\}$. Choose $x'\in h^+$  minimising $d(x_0,x')$. Let $m$ be the median of $x,x',x_0$. Then $m\in h^+$ since $x,x'$ are in $h^+$. Since $m$ lies on a geodesic from $x'$ to $x_0$, minimality of $x'$ implies that $m=x'$. Thus $x'$ lies on a geodesic from $x$ to $x_0$. However since $H_x=\{h\}$ the first edge of every geodesic from $x$ to $x_0$ crosses $h$. Thus we deduce that $x=x'$ and $d(x_0,x)$ is minimal.

For (2) let $k$ be a predecessor of $h$. As $y$ is adjacent to $h$, $k$ does not separate $y$ from $h$, thus $k$ must separate $x_0,y$. If $k$ is not adjacent to $y$ then there is some hyperplane $j$ adjacent to $y$ and separating $y$ from $k$. This hyperplane cannot meet $h$ otherwise it would also be adjacent to $x$ contradicting $H_x=\{h\}$. Since $j$ does not meet $h$ we deduce that $k<j<h$ contradicting the fact that $k$ is a predecessor. We thus deduce that the predecessors of $h$ are adjacent to $y$ and separate $y$ from $x_0$.

Conversely suppose $k$ is adjacent to $y$ and separates $y$ from $x_0$. As above it cannot meet $h$ as this would contradict $H_x=\{h\}$. We have $x\in h^+$, and since $k$ does not separate $x,y$ and $k^+$ contains $y$ we note that $k^+$ also contains $x$. Thus $h,k$ are not opposite. On the other hand $h$ separates $x_0,y$ while $k$ does not, so we deduce that $k<h$. If there were any $j$ with $k<j<h$ then $j$ would either separate $y$ from $h$ or from $k$. Since $y$ is adjacent to both $h,k$ this cannot happen. Hence we conclude that $k$ is a predecessor of $h$.
\end{proof}

\section{Controlled colourings}
\label{sec:colouring}

Let $X$ be a CAT(0) cube complex, and let $H$ denote the set of hyperplanes of $X$. In this section we introduce the concept of a controlled colouring on $H$, and prove that if $X$ is finite dimensional then controlled colourings exist. Throughout we fix a basepoint $x_0$ of $X$.

\begin{defn}
A directed edge from $x$ to $y$ is \emph{inward} if $d(x_0,y)<d(x_0,x)$. Equivalently, the hyperplane separating $x,y$ lies in the set $H_x$ of hyperplanes adjacent to $x$ and separating $x$ from $x_0$. An edge geodesic is \emph{inward} if each edge is inward.
\end{defn}

\begin{defn}
Given a colouring $c:H\to \{0,1\}$, a geodesic is \emph{monochromatic} if it only crosses hyperplanes of a single colour.
\end{defn}

\begin{defn}
An \emph{$l$-controlled colouring of $H$} is a map $c:H\to \{0,1\}$ such that no monochromatic inward geodesic has length greater than $l$.
\end{defn}

We will construct a colouring which is $l$-controlled where the constant $l$ depends only on the dimension of the CAT(0) cube complex $X$.

To construct the colouring we will introduce a collection of \emph{rank} functions on the hyperplanes. These in some sense measure the distance from a hyperplane to the basepoint $x_0$. Having established various important properties of these rank functions we will combine them into a rank-vector associated to each hyperplane, and it is this rank-vector which will allow us to define the colouring.

\begin{defn}
Let $d>1$. An \emph{outward (resp.\ inward) $d$-corner} in $X$ is an intersection $\cap_{i=1}^d h_i^+$ (resp.\ $\cap_{i=1}^d h_i^-$), where all the hyperplanes $h_i$ cross.
\end{defn}

By convention a $d$-corner means an outward $d$-corner unless otherwise specified.

\begin{defn}
The \emph{$d$-rank} of a hyperplane $h$ is defined inductively as follows. Let $H^d_{\geq 0}=H$. Having defined $H^d_{\geq n}$, define $H^d_{\geq n+1}$ to be the set of hyperplanes in $H^d_{\geq n}$ which are contained in some $d$-corner bounded by hyperplanes in $H^d_{\geq n}$. Define a hyperplane $h$ to have $d$-rank $n$ if $h$ lies in $H^d_{\geq n}\setminus H^d_{\geq n+1}$, that is $h\in H^d_{\geq n}$ but $h$ is not contained in any $d$-corner of hyperplanes in $H^d_{\geq n}$.
\end{defn}

In other words, $h$ has $d$-rank 0 if it is not contained in any $d$-corner; $h$ has $d$-rank 1 if it is contained in a $d$-corner of hyperplanes of $d$-rank 0, but not in any $d$-corner of hyperplanes which themselves are contained in a $d$-corner, etc. The example in Fig.\ \ref{fig2d} shows the 2-ranks of a negatively curved 2-dimensional CAT(0) cube complex.

\begin{figure}
\incgraphics{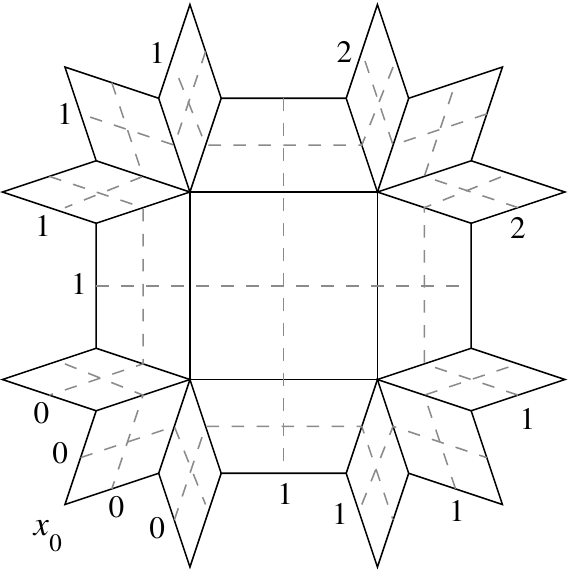}
\caption{2-ranks in the presence of negative curvature.}
\label{fig2d}
\end{figure}

As noted in section \ref{prelim}, associated to a subset $\Hhat$ of $H$ there is a CAT(0) cube complex $\Xhat$ which is a quotient of $X$, and whose hyperplanes are precisely $\Hhat$. An alternative way of formulating the definition of $d$-rank is to say that $h$ has $d$-rank 0 if it does not lie in any $d$-corner, and inductively $h$ has $d$-rank $n$ if it has $d$-rank 0 as a hyperplane of the quotient $X^d_{\geq n}$ whose hyperplanes are $H^d_{\geq n}$.

\begin{defn}
A $(d,n)$-corner is a $d$-corner bounded by hyperplanes of $d$-rank $n$.
\end{defn}

It is immediate from the definition that $\drank h=n$ implies $h$ lies in no $(d,n)$-corner. On the other hand if $n>0$ then consider a minimal $k$ in $H^d_{\geq n}$ with $k\leq h$. As $k\in H^d_{\geq n}$ it must lie in a $d$-corner of $H^d_{\geq n-1}$, and $h$ also lies in this corner. By minimality of $k$ the hyperplanes bounding this corner cannot lie in $H^d_{\geq n}$, so they are all of rank $n-1$. Thus if $\drank h=n>0$ then $h$ lies in a $(d,n-1)$-corner. Moreover applying this to the hyperplanes bounding the $(d,n-1)$-corner we conclude that for every $m<n$ there is a $(d,m)$-corner containing $h$.

\begin{lemma}
\label{drank-monot}
The $d$-rank is monotonic, that is for $k<h$ we have $\drank k \leq \drank h$.
\end{lemma}

\begin{proof}
By the above observations we note that the $d$-rank of a hyperplane is the maximal $n$ for which the hyperplane lies in a $(d,n-1)$-corner, or zero if there is no such $n$. If $k<h$ then every $d$-corner containing $k$ also contains $h$, hence the $d$-rank of $h$ is at least the $d$-rank of $k$.
\end{proof}

If $k$ is a predecessor of $h$ then $\drank k \leq \drank h$. In the special case that $k$ is the only predecessor of $h$ then we have equality. Indeed any $d$-corner containing $k$ also contains $h$, but conversely, for any $d$-corner containing $h$, each hyperplane $j$ bounding the corner must satisfy $j< k< h$. Thus a $d$-corner contains $k$ if and only if it contains $h$.

\begin{defn}
A CAT(0) cube complex $X$ is \emph{$d$-flat} if every hyperplane of $X$ has $d$-rank 0. Equivalently $X$ is $d$-flat if no hyperplane of $X$ lies in a $d$-corner.
\end{defn}

The term `flat' is motivated by the observation that $\R^n$ with the standard cubing is $d$-flat for all $d\geq 2$. Note also that $X$ is $d$-flat for all $d$ greater than the dimension of $X$.

\begin{defn}
The \emph{flatness} of $X$ is the least $d$ such that $X$ is $(d+1)$-flat.
\end{defn}

If $X$ has dimension $D$ then it is $(D+1)$-flat, hence the flatness is bounded above by the dimension.

\begin{prop}
\label{flat-pred}
If $X$ is $(d+1)$-flat then each hyperplane $h$ of $X$ has at most $d$ predecessors and every predecessor has $\drank$ at least $\drank(h)-1$.
\end{prop}

\begin{proof}
If $X$ is $(d+1)$-flat then no hyperplane lies in a $(d+1)$-corner, so in particular no hyperplane can have more than $d$ predecessors.

Let $n=\drank h$. If $n=0$ then there is nothing to prove. Otherwise $h$ lies in some $(d,n-1)$-corner $S$. If $k$ is a predecessor of $h$ then no hyperplane $j$ bounding $S$ can separate $h,k$. If for some $j$ we have $j< k$ or $j=k$ then by monotonicity we deduce that $\drank k \geq n-1$ as required. This leaves the possibility that $k$ crosses every $j$ bounding $S$. But then intersecting $S$ with $k^+$ we obtain a $(d+1)$-corner containing $h$. This contradicts $(d+1)$-flatness.
\end{proof}

\begin{defn}
A \emph{chain} in a path $s$ is a set of hyperplanes meeting $s$, which is totally ordered by $<$.
\end{defn}

\begin{defn}
An inward geodesic $s$ is \emph{straight} if the set of hyperplanes crossed by the geodesic is a chain in $s$, i.e.\ the hyperplanes are pairwise disjoint.
\end{defn}

If $h$ has more than one predecessor, then there is an \emph{inward} corner associated to $h$, \emph{viz}.\ the intersection $k_1^-\cap\dots\cap k_p^-$ where $k_1,\dots,k_p$ are the predecessors of $h$. Let $\i_h$ denote this inward corner. By convention we take $\i_h$ to be empty when $h$ has at most one predecessor.

Note that if $s$ is a straight path then $s$ is disjoint from $\i_h$ for each $h$ meeting $s$. Indeed for $s$ to have non-empty intersection with $\i_h$ then in particular it must cross every predecessor of $h$, contradicting straightness. Thus the following definition gives a generalisation of straightness.

\begin{defn}
An inward geodesic $s$ is \emph{bound to $h$} if it is disjoint from $\i_h$, and is \emph{bound} if it is bound to $h$ for all $h$ meeting $s$.
\end{defn}

If $s$ is bound to $h$ and $h$ has more than one predecessor then there is at least one predecessor which $s$ does not cross. Conceptually, $s$ is bound to $h$ if there is at least one direction in which $s$ goes no further towards $x_0$ than the vertex $x$ adjacent to $h$ minimising $d(x_0,x)$. See Fig.\ \ref{figbound}.

\begin{figure}
\incgraphics{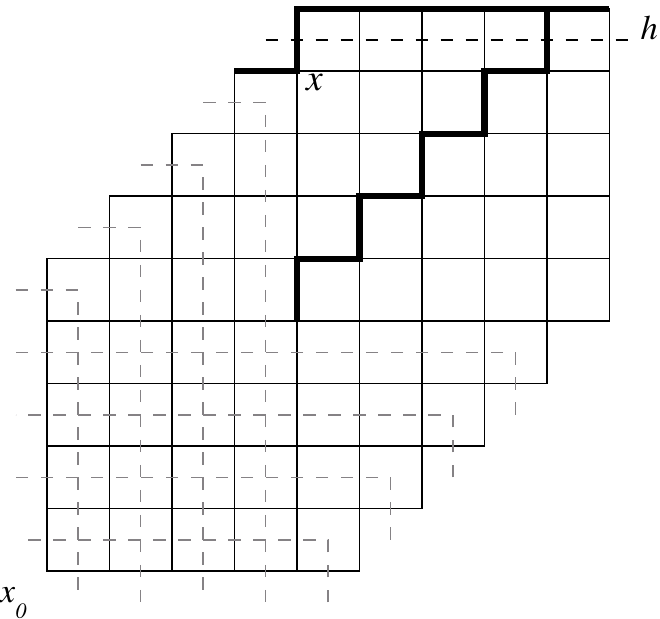}
\caption{Bound geodesics and the corners they are disjoint from.}
\label{figbound}
\end{figure}

Ultimately we wish to construct a colouring for which the length of monochromatic geodesics is bounded. Our colouring will have the property that monochromatic geodesics are bound, so we will control the length of these by first controlling the variation in rank that can occur on a bound geodesic.

\begin{prop}
\label{straight-path}
\label{q-straight-path}
Let $X$ be a $(d+1)$-flat CAT(0) cube complex. If $s$ is a bound inward geodesic and $C$ is a chain in $s$ then any two hyperplanes in $C$ have $d$-ranks differing by at most 2. In particular a straight inward geodesic in $X$ crosses hyperplanes whose $d$-ranks differ by at most 2.
\end{prop}

\begin{proof}
Without loss of generality we may assume that $C$ is a maximal chain in $s$. Hence enumerating $C$ as $h_0>h_1>\dots>h_p$, for $i=0,1,\dots,p-1$, the hyperplane $h_{i+1}$ is a predecessor of $h_i$. Let $n$ be the $d$-rank of $h_0$. By monotonicity, each $h_i$ has $d$-rank at most $n$, thus we may assume $n>2$, otherwise there is nothing to prove.

If for each $i=0,1,\dots,p-1$, $h_{i+1}$ is the unique predecessor of $h_i$ then every hyperplane in $C$ has rank $n$ and we are done. Otherwise, consider the least $i$ such that $h_i$ has more than one predecessor, and note that $h_i$ has $d$-rank $n$. As $s$ is bound, there exists a predecessor $k$ of $h_i$, with $k$ not meeting $s$. By Proposition \ref{flat-pred} the $d$-rank of $k$ is at least $n-1$, hence $k$ lies in a $(d,n_i-2)$-corner $S$. We note that $h_i$ lies in $k^+$ which is contained in $S$, so $h_i$ also lies in the $d$-corner $S$.

Now suppose there exists a hyperplane $l$ in $C$ with $\drank l \leq n_i-3$. Let $j$ be one of the hyperplanes bounding $S$. We cannot have $j<l$ by monotonicity, since $\drank j=n-2$ and $\drank l\leq n-3$. On the other hand, since $k$ does not meet $s$ we know that $s$ lies in $k^+$ which is contained in $j^+$. Hence $l$ meets $j^+$, but is not contained in it, so $l$ and $j$ must intersect. Since this holds for every $j$ bounding $S$, we have a $(d+1)$-corner $S\cap l^+$ containing $h_i$. This contradicts $(d+1)$-flatness, so we conclude that every element of $C$ has $d$-rank at least $n-2$.

In the case where $s$ is straight, the set $C$ of all hyperplanes meeting $s$ is a chain, hence hyperplanes crossing $s$ have $d$-ranks differing by at most 2.
\end{proof}

\medskip
Let $D$ be the dimension of $X$. For a hyperplane $h$ the $d$-ranks for $d=D,D-1,\dots, 2$ give progressively finer information about $h$. To make this precise, we associate to each hyperplane a \emph{rank-vector}.

Starting with $X$, let $H_{(n_D)}$ denote the set of hyperplanes of $X$ with $D$-rank $n_D$. Associated to this is a quotient $X_{(n_D)}$ of $X$, whose hyperplanes are $H_{(n_D)}$. This is $D$-flat by construction, since any hyperplane lying in a $D$-corner of $D$-rank $n_D$ hyperplanes must have $D$-rank greater than $n_D$. We note that $H=\bigsqcup\limits_{n_D} H_{(n_D)}$.

Now let $H_{(n_D,n_{D-1})}$ denote the set of hyperplanes of $X_{(n_D)}$ with $(D-1)$-rank $n_{D-1}$, and let $X_{(n_D,n_{D-1})}$ be the associated $(D-1)$-flat complex. We have $H_{(n_D)}=\bigsqcup\limits_{n_{D-1}} H_{(n_D,n_{D-1})}$ and hence $H=\bigsqcup\limits_{n_D,n_{D-1}} H_{(n_D,n_{D-1})}$. We repeat this process to produce $H_{(n_D,n_{D-1},\dots, n_2)}$ and $X_{(n_D,n_{D-1},\dots, n_2)}$, with $H=\bigsqcup\limits_{n_D,n_{D-1},\dots, n_2} H_{(n_D,n_{D-1},\dots, n_2)}$.

\begin{defn}
The \emph{rank-vector} of $h$ is $\r(h)=(n_D,\dots, n_2)$ if $h$ lies in $H_{(n_D,\dots, n_2)}$.
\end{defn}

Figure \ref{figpgl} illustrates the concept of rank vectors for the example of the $PGL_2(\Z)$ CAT(0) cube complex from \cite{CN}.

\begin{figure}
\incgraphics{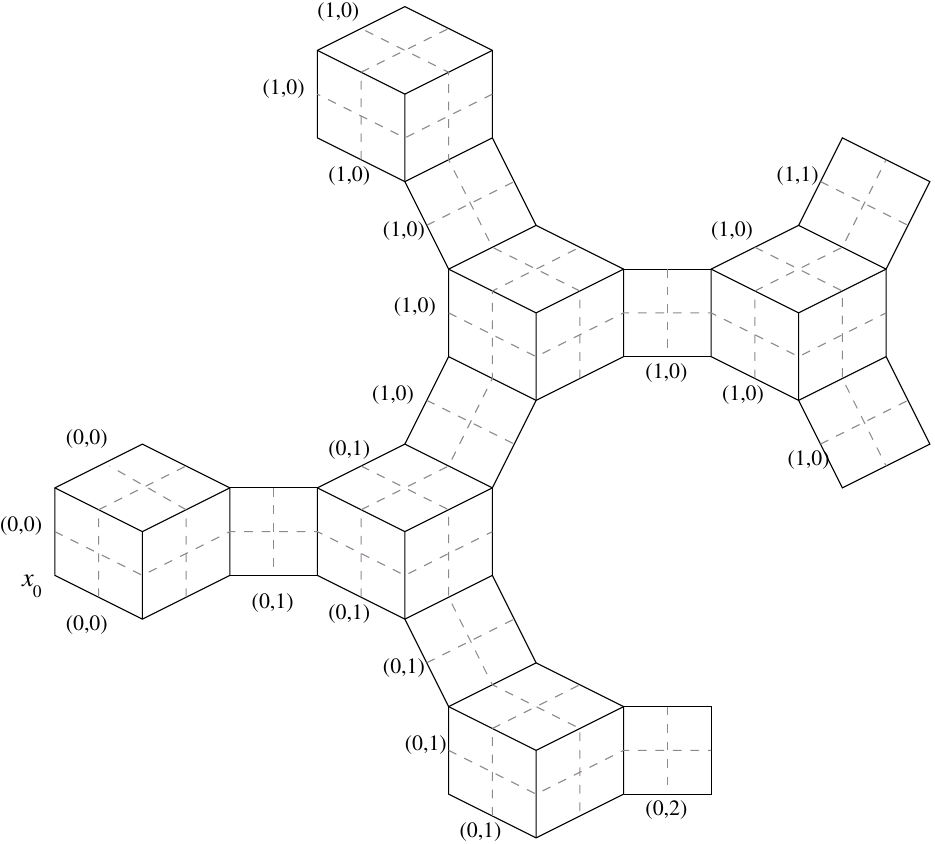}
\caption{Rank vectors for the $PGL_2(\Z)$ CAT(0) cube complex from \cite{CN}.}
\label{figpgl}
\end{figure}

We equip the set of rank-vectors with the lexicographic order.

\begin{lemma}
\label{rv-monot}
The rank-vector function is monotonic.
\end{lemma}

\begin{proof}
If $k<h$ then $D\rnk k \leq D\rnk h$. The $D$-rank is the first entry of the rank vector so if the inequality is strict then the the rank-vectors satisfy the same inequality. Otherwise $k,h$ lie in the same $H_{(n_D)}$, and as $h<k$ we have $(D-1)\rnk k \leq (D-1)\rnk h$ in $H_{(n_D)}$. Repeating the argument we either have strict inequality at some point, giving a strict inequality of rank-vectors, or we conclude that $k,h$ lie in the same $H_{(n_D,\dots, n_2)}$, and we have equality of rank vectors.
\end{proof}

\begin{lemma}
\label{unique-max}
A hyperplane $h$ has at most one predecessor with the same rank-vector.
\end{lemma}

\begin{proof}
Let $(n_D,\dots, n_2)$ be the rank-vector of $h$. Then $h$ is a hyperplane of $X_{(n_D,\dots, n_2)}$. If $k$ is a predecessor of $h$ of the same rank then $k$ is also a hyperplane of $X_{(n_D,\dots, n_2)}$, and indeed must be a predecessor of $h$ in $X_{(n_D,\dots, n_2)}$. But $X_{(n_D,\dots, n_2)}$ is 2-flat, so $h$ has at most 1 predecessor in $X_{(n_D,\dots, n_2)}$.
\end{proof}

\medskip
We are now in a position to define the colouring $c$. We will require two key properties of the colouring.
\begin{enumerate}
\item If $h$ has a predecessor $k$ with $c(h)=c(k)$ then $h$ also has a predecessor $j$ with $c(h)\neq c(j)$.

\item If $h$ has a predecessor $k$ with $\r(h)=\r(k)$ then $c(h)\neq c(k)$.
\end{enumerate}
This is achieved by the following definitions.

\begin{defn}
A predecessor $k$ of $h$ is \emph{$\r$-maximal} if $\r(j)\leq \r(k)$ for all predecessors $j$ of $h$.
\end{defn}

\begin{defn}
The colouring $c(h)$ of a hyperplane $h$ is defined inductively to be
$$\begin{cases}
1 & \text{if all $\r$-maximal predecessors of $h$ are coloured 0}\\
0 & \text{otherwise}.
\end{cases}$$
\end{defn}

The fact that this colouring satisfies (1) above is immediate: if all predecessors of $h$ are coloured 0 then $h$ will be coloured 1, while if all predecessors of $h$ are coloured 1 then $h$ will be coloured 0. To see that (2) is satisfied, if $h$ has a predecessor $k$ with $\r(h)=\r(k)$ then by monotonicity this predecessor is $\r$-maximal. By Lemma \ref{unique-max} there cannot be more than one predecessor having the same rank vector so $k$ is the unique $\r$-maximal predecessor of $h$. Thus $h$ is coloured 1 if $k$ is coloured 0, and $h$ is coloured 0 otherwise.

\medskip

We will now show that monochromatic paths for this colouring are bound in the following strong sense.

\begin{defn}
An inward geodesic $s$ is \emph{totally bound} if $s$ is bound, and the quotient of $s$ in $X_{(n_D,n_{D-1},\dots,n_d)}$ is bound for all $d$, and all $n_D,n_{D-1},\dots,n_d$.
\end{defn}

\begin{prop}
\label{mono-qs}
Let $s$ be a monochromatic inward geodesic in $X$. Then $s$ is totally bound.
\end{prop}

\begin{proof}
We first show that $s$ is bound. Let $h$ be a hyperplane meeting $s$. As $s$ is monochromatic, every predecessor of $h$ meeting $s$ must have the same colour as $h$. But by construction of the colouring, $h$ cannot have the same colour as all of its predecessors, hence each $h$ meeting $s$ has a predecessor not meeting $s$ as required.

To see that $s$ is totally bound, consider the quotient $s'$ of $s$ in $X_{(n_D,n_{D-1},\dots,n_d)}$ for some $d, n_D,n_{D-1},\dots,n_d$. Suppose $h$ meets $s'$, that is $h$ meets $s$ and $h\in H_{(n_D,n_{D-1},\dots,n_d)}$. If $h$ has a predecessor $k$ meeting $s'$, then the maximal rank-vector of a predecessor of $h$ must be between $\r(h)$ and $\r(k)$, hence it is of the form $(n_D,n_{D-1},\dots,n_d,?,\dots,?)$. In particular the $\r$-maximal predecessors of $h$ all lie in $H_{(n_D,n_{D-1},\dots,n_d)}$. By construction of the colouring the $\r$-maximal predecessors of $h$ cannot all have the same colour as $h$, thus as $s$ is monochromatic there is an $\r$-maximal predecessor of $h$ not meeting $s$ and hence not meeting $s'$. We conclude that for each $h\in H_{(n_D,n_{D-1},\dots,n_d)}$ meeting $s'$ there is a predecessor of $h$ in $H_{(n_D,n_{D-1},\dots,n_d)}$ not meeting $s'$ as required.
\end{proof}

In Proposition \ref{q-straight-path} we established bounds on the number of different $d$-ranks that can occur on a bound geodesic. We will now prove a version of this for rank-vectors.

\begin{prop}
\label{q-straight-non-flat}
Let $s$ be a totally bound inward geodesic in $X$ and let $C$ be a chain in $s$. Then the rank-vector function takes at most $3^{f-1}$ different values on $C$, where $f$ is the flatness of $X$.
\end{prop}

\begin{proof}
We prove this by induction on the flatness of $X$. If $X$ has flatness 1, i.e.\ it is 2-flat then every rank-vector is zero and there is nothing to prove.

For the induction step, if $X$ has flatness $f$, then the first non-zero entry, $n_f$, of the rank-vector is precisely the $f$-rank. Since $X$ is $(f+1)$-flat, $n_f$ takes at most 3 values by Proposition \ref{q-straight-path}. For each of these values of $n_f$, consider the quotient $s_{(n_f)}$ of $s$ in $X_{(n_f)}$, and let $C_{(n_f)}$ be the intersection of $C$ with $H_{(n_f)}$. Then $s_{(n_f)}$ is a totally bound inward geodesic in $X_{(n_f)}$ and $C_{(n_f)}$ is a chain in $s_{(n_f)}$. The quotient $X_{(n_f)}$ is $f$-flat, so has flatness at most $f-1$. Hence by induction the rank-vector takes at most $3^{f-2}$ values on each $C_{(n_f)}$, for each of the 3 values of $n_f$. Hence it takes at most $3^{f-1}$ values on $C$.
\end{proof}

We are now ready to state and prove the colouring theorem. From Propositions \ref{mono-qs} and \ref{q-straight-non-flat} we know that monochromatic geodesics are totally bound and that chains in these geodesics therefore have a bounded number of rank-vectors. We combine this with an idea from \cite{BCGNW} to prove the following.

\begin{thm}
\label{colouring}
Let $X$ be a CAT(0) cube complex of dimension $D$ and flatness $f$. Then the colouring $c$ defined above is a $3^{f-1}D$-controlled colouring of $H$. That is, no monochromatic inward geodesic can have length greater than $3^{f-1}D$.
\end{thm}

\begin{proof}
Let $s$ be a monochromatic inward geodesic. Let $x,y$ be the initial and final vertices of $s$. The set of hyperplanes meeting $s$ is precisely the set of hyperplanes crossing the interval $[x,y]$. In \cite{BCGNW} it is shown that an interval in a CAT(0) cube complex can be isometrically embedded as a subcomplex of the Euclidean space of the same dimension. In particular the set of hyperplanes crossing an interval of dimension at most $D$ can be partitioned into $D$ chains. Thus we have chains $C_1,\dots C_D$ partitioning the hyperplanes that meet $s$.

As $s$ is monochromatic, by Proposition \ref{mono-qs} it is totally bound. Hence by Proposition \ref{q-straight-non-flat} for any chain $C$ in $s$ the rank-vector function takes at most $3^{f-1}$ different values on $C$. For each $i$ we apply this to a maximal chain $C_i'$ in $s$, containing $C_i$. Enumerate $C_i'$ as $h_0>h_1>\dots>h_p$. For each $j$, $h_{j+1}$ is a predecessor of $h_j$, by maximality, so as $c(h_i)=c(h_{i+1})$ we know that $\r(h_{j+1})<\r(h_j)$. Thus the rank-vector function takes $|C_i'|$ different values on $C_i'$, so $|C_i'|\leq 3^{f-1}$.

We conclude that there are $D$ chains $C_1,\dots, C_D$ partitioning the hyperplanes that meet $s$, and as $C_i\subseteq C_i'$, each chain $C_i$ has cardinality at most $3^{f-1}$. Hence there are at most $3^{f-1}D$ hyperplanes meeting $s$, that is $s$ has length at most $3^{f-1}D$. This completes the proof.
\end{proof}

\section{The projection theorem}
\label{sec:projection}

In this section we will consider coordinates on a CAT(0) cube complex $X$. These will give an embedding of $X$ into $l^1(H)$ where $H$ is the set of hyperplanes of $X$ (cf.\ \cite{NRoll}). The main result of this section is a projection theorem which allows us to retract from a cube in $l^1(H)$ onto the embedded copy of $X$.

Let $\C=\C_H$ denote the set of all finitely supported elements $\xi=(\xi_h)$ of $l^1(H)$ such that $\xi_h\in[0,1]$ for all $h\in H$. Viewing $\C$ as the union of the finite dimensional cubes corresponding to finite subsets of $X$, we see that $\C$ is a CAT(0) cube complex and indeed there is a canonical bijection between the hyperplanes of $\C$ and the hyperplanes $H$ of $X$. The restriction of the metric on $l^1(H)$ to $\C$ gives the $l^1$ path metric on $\C$. The vertex set of $\C$ is precisely the finitely supported $\{0,1\}$-valued functions on $H$.

The CAT(0) cube complex $X$ can be embedded into $\C$ as follows. For a vertex $x$ in $X$ we map $x$ to the vertex $\xi$ of $\C$ defined by $\xi_h=1$ for all $h$ separating $x_0,x$ and $\xi_h=0$ otherwise. We extend the map affinely to give a map from $X$ to $\C$. The bijection between the hyperplanes of $\C$ and the hyperplanes $H$ of $X$ makes this map isometric on the vertex set (with edge-path metric) since the distance between two vertices of $X$ is the number of hyperplanes separating them. We will see that it is isometric on the whole of $X$ (equipped with the $l^1$ path metric) as a corollary of the projection theorem.

It is often convenient to identify $X$ with its image in $\C$. Indeed one can think of the embedding as providing coordinates $(\xi_h)$ associated to a point $x$ in $X$.

\begin{lemma}
\label{image-of-X}
Let $F$ be a finite set of hyperplanes in $H$. A vertex $\xi=\chi_F$ of $\C$ lies in the image of $X$ if and only if
\begin{enumerate}
\item when $k\in F$ and $h<k$ then $h\in F$, and
\item $F$ contains no opposite pair of hyperplanes.
\end{enumerate}
\end{lemma}
\begin{proof}
If $\xi=\chi_F$ lies in the image of $X$ then $F$ is the set of hyperplanes separating $x_0,x$ for some $x\in X$. Thus $k\in F$ and $h<k$ implies that $h^+\supset k^+\ni x$, so $h\in F$. Similarly $h,k\in F$ implies $h^+\cap k^+$ is non-empty (it contains $x$) so $h,k$ are not opposite.

Conversely suppose that $F$ is a finite set of hyperplanes in $H$ satisfying (1),(2). For each $h$ in $H$ we choose the orientation $h^+$ if $h\in F$, and $h^-$ otherwise. Finitely many of these orientations differ from $\h_{x_0}=\{h^- : h\in H\}$. Moreover we will show that these halfspaces are pairwise intersecting. If $h,k\in F$ then $h^+\cap k^+$ is non-empty as $h,k$ are not opposite, while if $h,k\in H\setminus F$ then $h^+\cap k^+$ is non-empty as it contains $x_0$. On the other hand if one of $h,k$ lies in $F$ and the other in $H\setminus F$, say $k\in F$ and $h\in H\setminus F$ then by (1) we cannot have $h<k$. That is $h^+$ does not contain $k^+$ so that $h^-\cap k^+$ is non-empty. Hence the orientations determine a vertex of $X$, and $\xi=\chi_F$ is the image of this vertex.
\end{proof}


\begin{defn}
Let $h,k$ be opposite hyperplanes. A point $\xi$ in $\C$ is \emph{$(h,k)$-intervalic} if one or both of $\xi_h,\xi_k$ is zero. We say that $\xi$ is \emph{intervalic} if it is $(h,k)$-intervalic for all opposite pairs $h,k$.
\end{defn}

\begin{defn}
Let $h,k$ be hyperplanes with $h<k$. A point $\xi$ in $\C$ is \emph{$(h,k)$-actual} if $x_h=1$ or $x_k=0$. Otherwise we say that $\xi$ is \emph{$(h,k)$-virtual}. We say that $\xi$ is \emph{actual} if it is $(h,k)$-actual for all pairs $h<k$.
\end{defn}

In terms of these definitions, Lemma \ref{image-of-X} amounts to the statement that if $\xi\in \C$ is $\{0,1\}$-valued then $\xi$ is in the image of $X$ if and only if $\xi$ is intervalic and actual. We now extend this to general points of $\C$.

\begin{lemma}
\label{image-of-X2}
A point $\xi$ in $\C$ lies in the image of $X$ if and only if $\xi$ is intervalic and actual.
\end{lemma}

\begin{proof}
First suppose that $\xi$ in lies in the image of $X$. Let $F=\{h\in H : 0<\xi_h<1\}$. Then $\xi$ lies in an $|F|$ dimensional subcube of $\C$: the $2^{|F|}$ vertices of this cube are obtained by letting $\eta_h$ be either 0 or 1 for each $h$ in $F$, and setting $\eta_h=\xi_h$ for $h\notin F$. Since $\xi$ lies in the image of $X$, these vertices must also lie in the image of $X$.

If $\xi_h,\xi_k$ are non-zero then there is a vertex $\eta$ such that $\eta_h=\eta_k=1$. As this vertex is in the image of $X$ we deduce that $h,k$ cannot be opposite, hence $\xi$ is intervalic. Similarly if $\xi_h<1$ and $\xi_k>0$ then there is a vertex $\eta$ such that $\eta_h=0$ and $\eta_k=1$. Thus we cannot have $h<k$, so $\xi$ is actual.

Conversely, suppose that $\xi$ is intervalic and actual. Let $F=\{h\in H : 0<\xi_h<1\}$ and consider the vertices obtained by letting $\eta_h$ be 0 or 1 for each $h$ in $F$, and setting $\eta_h=\xi_h$ for $h\notin F$. Once again these span a subcube of $\C$ containing $\xi$. These vertices are intervalic and actual: if $\eta_h,\eta_k$ are non-zero then the same is true of $\xi_h,\xi_k$ so that $h,k$ cannot be opposite, and if $\eta_h=0$ and $\eta_k=1$ then $\xi_h<1$ and $\xi_k>0$ so we cannot have $h<k$. We deduce that the vertices of the subcube are all in the image of $X$. Moreover the correspondence between the hyperplanes of $X$ and those of $\C$ guarantees that the vertices in $X$ whose images are the vertices of the subcube will themselves span a cube in $X$. Hence the subcube of $\C$, and in particular the point $\xi$ are in the image of $X$.
\end{proof}

To prove the projection theorem we will need to consider infinite compositions of functions. Let $S$ be a set, and $(P_\lambda)_{\lambda\in \Lambda}$ a family of maps from $S$ to itself. For $s\in S$ we say that the \emph{$\Lambda$-support} of $s$, denoted $\LSup(s)$, is the set of $\lambda\in \Lambda$ such that $P_\lambda$ does not fix $s$. If $\Lambda$ is equipped with a total order $\prec$, and $F$ is a finite subset of $\Lambda$ then define $P_F=P_{\lambda_k}\circ\dots \circ P_{\lambda_1}$ where $F$ is enumerated as $\lambda_1,\dots,\lambda_k$, ordered by $\prec$.

\begin{lemma}
\label{infinite-comp}
 Let $S$ be a set, and $(P_\lambda)_{\lambda\in \Lambda}$ a family of maps from $S$ to itself. Suppose that
\begin{enumerate}
\item for all $s\in S$ the $\Lambda$-support of $s$ is finite,

\item for all $s$ in $S$ and for $\lambda$ the $\prec$-least element of $\LSup(s)$, we have
$$\LSup(P_\lambda(s))\subseteq \LSup(s)\setminus\{\lambda\}.$$
\end{enumerate}
Then for $F$ a finite subset of $\Lambda$ containing $\LSup(s)$, the $\Lambda$-support of $P_F(s)$ is empty. Moreover $P_F(s)$ does not depend on the choice of $F$ containing $\LSup(s)$, and if $\LSup(s)$ is empty then $P_F(s)=s$ for all $F$.

If $S$ is a metric space, and each $P_\lambda$ is contractive then the map $P:S\to S$ defined by $s\mapsto P_F(s)$ for any $F$ containing $\LSup(s)$, is also contractive.
\end{lemma}

The map $P$ can be regarded as the infinite composition of all $P_\lambda$ ordered by $\prec$.

\begin{proof}
Enumerate $F$ in $\prec$-order as $\lambda_1,\lambda_2,\dots,\lambda_k$. Let $P_0(s)=s$ and for $i=1,\dots,k$, let $P_i(s)=P_{\lambda_i}(P_{i-1}(s))$. We claim that for each $i$ the $\Lambda$-support of $P_i(s)$ is contained in $F_i=\{\lambda_{i+1},\dots,\lambda_k\}\cap\LSup(s)$, hence in particular when $i=k$ we find that the $\Lambda$-support is empty.

For $i=0$ this is true by hypothesis. Now supposing that $\LSup(P_{i-1}(s))$ is contained in $F_{i-1}=\{\lambda_i,\dots,\lambda_k\}\cap\LSup(s)$ then either $\lambda_i$ is the least element of $\LSup(P_{i-1}(s))$ or $\lambda_i$ is not in the $\Lambda$-support of $P_{i-1}(s)$. In the first case we have
$$\LSup(P_{\lambda_i}(P_{i-1}(s)))\subseteq \LSup(P_{i-1}(s))\setminus\{\lambda_i\}\subseteq \{\lambda_{i+1},\dots,\lambda_k\}\cap\LSup(s),$$
while in the second case $P_{\lambda_i}(P_{i-1}(s))=P_{i-1}(s)$, so the $\Lambda$-support is unchanged. By the induction hypothesis $\LSup(P_{i-1}(s))$ is contained in $F_{i-1}$, but in this second case $\lambda_i$ is not in the $\Lambda$-support, so we deduce that $\LSup(P_i(s))$ is contained in $F_i$ as required.

We note that if $\lambda_i$ is not in the $\Lambda$-support of $s$, then it is not in the $\Lambda$-support of $P_{i-1}(s)$, and we have $P_i(s)=P_{i-1}(s)$. Hence in the composition we have 
$$P_F(s)=P_{\lambda_k}\circ \dots \circ P_{\lambda_{i+1}}\circ P_{\lambda_i}(P_{i-1}(s))=P_{\lambda_k}\circ \dots \circ P_{\lambda_{i+1}}(P_{i-1}(s)).$$
Hence $P_F(s)=P_{F\setminus\{\lambda_i\}}(s)$. Removing all such $\lambda_i$ we deduce that $P_F(s)=P_{\LSup(s)}(s)$, so is independent of the choice of $F$. In the special case that $\LSup(s)$ is empty, each $P_{\lambda_i}$ fixes $s$ and we have $P_F(s)=s$.

Now suppose that $S$ is a metric space and each $P_{\lambda}$ is contractive. Given $s,t\in S$, take $F$ to be a finite set containing both $\LSup(s)$ and $\LSup(t)$. Then $P(s)=P_F(s)$ and $P(t)=P_F(t)$, so as $P_F$ is contractive we deduce that $d(P(s),P(t))\leq d(s,t)$. Since this is true for any pair $s,t$ we conclude that $P$ is contractive.
\end{proof}

Before we prove the projection theorem, we examine a couple of simple examples. First consider $X_1$ a tree of three edges emanating from the basepoint $x_0$. We identify $H_1$ with $\{1,2,3\}$ and hence $l^1(H_1)=\R^3$ with the $l^1$-metric. The space $X_1$ is embedded in $\R^3$ as intervals along the three axes. In this example the image of $X_1$ is precisely the set of intervalic points in the cube (all intervalic points are actual). We project from the cube in $\R^3$ onto $X_1$ (i.e. the set of intervalic points) by first projecting onto two faces of the cube and then for each face projecting onto two edges of the face, as in Fig.\ \ref{figproj} (a). The choice of which direction to project in first affects the map, but each choice will produce a retraction onto $X_1$.

For the second example, consider $X_2$ a line of three edges starting from the basepoint $x_0$. Again we have $l^1(H_2)=\R^3$ with the $l^1$-metric, and the space $X_2$ is embedded in the cube as illustrated in Fig.\ \ref{figproj} (b)(i). In this case the cube is intervalic, and the image of $X_2$ is the actual points. For this example, the order of the projections is important. We must first project in the $(1,3)$-plane, and then do the other two projections as in Fig.\ \ref{figproj} (b)(i). Doing the projections in a different order we do not obtain the correct range. For example, doing $(1,2)$ then $(2,3)$, projects onto the union of a square and an edge on which the $(1,3)$ projection is the identity, see Fig.\ \ref{figproj} (b)(ii).

\begin{figure}
\incgraphics{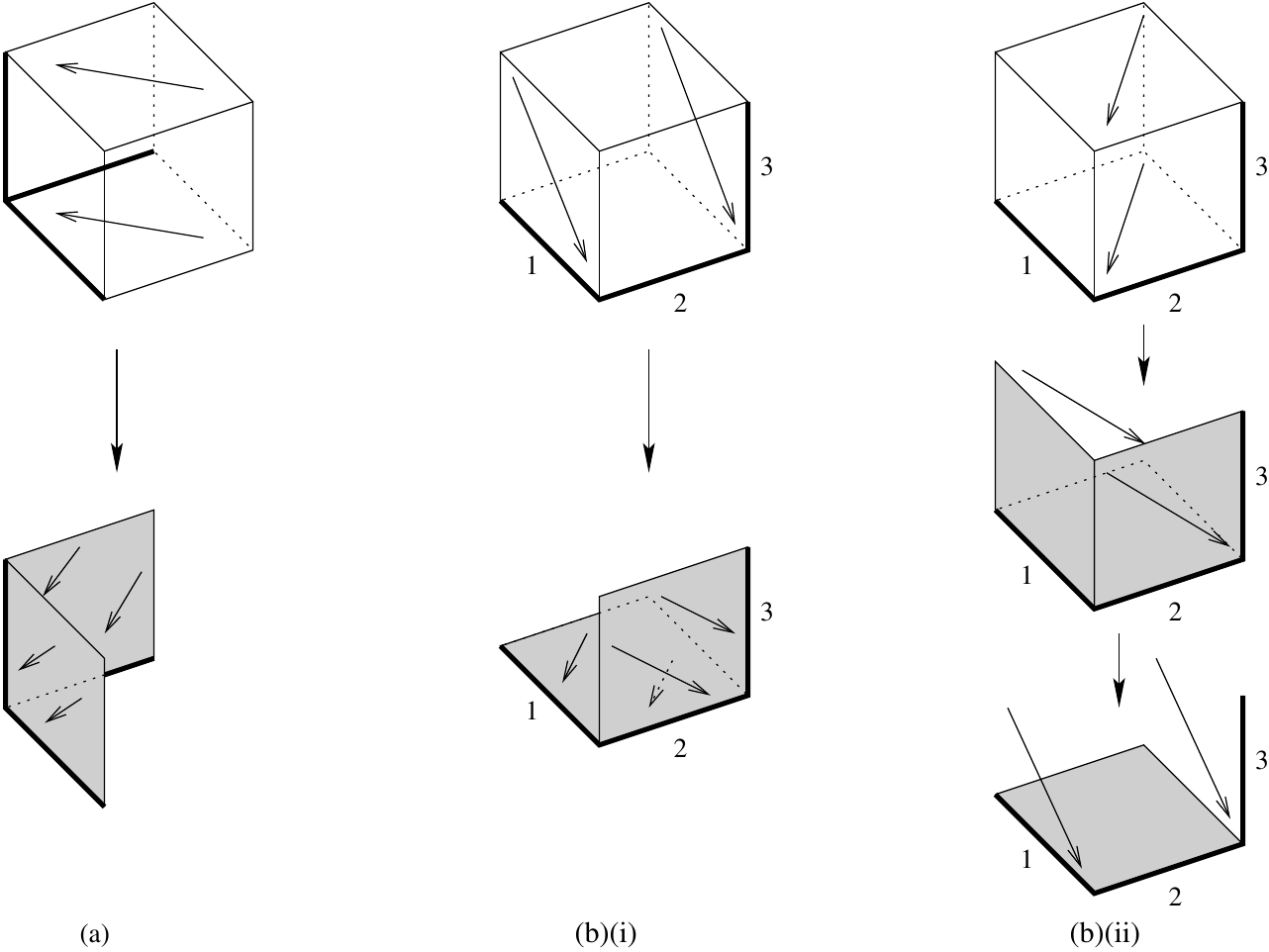}
\caption{Two examples of the Projection Theorem.}
\label{figproj}
\end{figure}

With these two examples in mind, we will now prove the theorem.

\begin{thm}[Projection Theorem]
\label{projection-theorem}
Let $X$ be a CAT(0) cube complex, let $H$ be the set of hyperplanes of $X$ and let $\C$ be the cube of finitely supported elements $\xi\in l^1(H)$ with $\xi_h\in[0,1]$ for all $h\in H$. Embed $X$ into $\C$ as above and let $A$ denote the image of $X$ in $\C$. Then there is a map $P:\C\to A$ such that
\begin{enumerate}
\item $P$ is contractive,
\item the restriction of $P$ to $A$ is the identity,
\item if $\xi\in \C$ is intervalic then $\|P(\xi)\|_1=\|\xi\|_1$.
\end{enumerate}
\end{thm}

\begin{proof}
We note that $A$ is precisely the set of intervalic and actual points in $\C$. Let $I$ denote the set of intervalic points of $\C$. We will construct the projection $P$ as a composition $P=P_A\circ P_I$, where both $P_A,P_I$ are contractive, $P_I$ is a retraction of $\C$ onto $I$, $P_A$ is a retraction of $I$ onto $A$, and $P_A$ preserves the $l^1$ norm.

Let $\Hop=\{\{h,k\}:h,k\in H,\, \text{$h$ opposite $k$}\}$. We define a map $\pop$ on $[0,1]^2$ by
$$\pop(\xi_1,\xi_2)=\begin{cases}
(\xi_1-\xi_2,0) &\text{if } \xi_1\geq \xi_2\\
(0,\xi_2-\xi_1) &\text{if } \xi_1\leq \xi_2
\end{cases}.$$
For $\{h,k\}$ in $\Hop$ we then define a map $P^{h,k}$ on $\C$. Let $P^{h,k}(\xi)=\xi'$ where $(\xi'_h,\xi'_k)=\pop(\xi_h,\xi_k)$ and $\xi'_j=\xi_j$ for $j\neq h,k$.

To show that each $P^{h,k}$ is contractive it suffices to show that $\pop$ is contractive. Take $(\xi_1,\xi_2),(\eta_1,\eta_2)$ in $[0,1]^2$. If $\xi_1\geq \xi_2$ and $\eta_1\geq \eta_2$ then
\begin{align*}
\normi{\pop(\xi_1,\xi_2)-\pop(\eta_1,\eta_2)}&=\normi{(\xi_1-\xi_2)-(\eta_1-\eta_2),0}\\
&\leq|\xi_1-\eta_1|+|\xi_2-\eta_2|=\normi{(\xi_1-\eta_1,\xi_2-\eta_2)}.
\end{align*}
Similarly if $\xi_1\leq \xi_2$ and $\eta_1\leq \eta_2$. Now suppose that $\xi_1\geq \xi_2$ and $\eta_1\leq \eta_2$. Then 
\begin{align*}
\normi{\pop(\xi_1,\xi_2)-\pop(\eta_1,\eta_2)}&=\normi{((\xi_1-\xi_2)-0,0-(\eta_2-\eta_1))}=\xi_1-\xi_2+\eta_2-\eta_1\\
&\leq |\xi_1-\eta_1|+|\xi_2-\eta_2|=\normi{(\xi_1-\eta_1,\xi_2-\eta_2)}.
\end{align*}
We conclude that $\pop$ is contractive, hence each map $P^{h,k}$ is contractive.

To construct $P_I$ we apply Lemma \ref{infinite-comp}, for the family $P^{h,k}, \{h,k\}\in \Hop$. We note that the $\Hop$-support of a point $\xi$ is the set of all non-intervalic pairs $\{h,k\}$ for $\xi$. This is finite, since if $\{h,k\}$ is in the $\Hop$-support of $\xi$ then $h$ and $k$ are in the usual support of $\xi$ as a function on $H$. Note that each $P^{h,k}$ reduces coordinates, so for any $h,k$, a point which is $(h',k')$-intervalic will remain so upon applying $P^{h,k}$. In other words applying $P^{h,k}$ does not increase $\Hop$-supports. The point $P^{h,k}(\xi)$ is $(h,k)$-intervalic by construction, hence we deduce that for any $\lambda=\{h,k\}$
$$\HopSup(P^\lambda(s))\subseteq \HopSup(s)\setminus\{\lambda\}.$$
Thus choosing any ordering on $\Hop$, we can apply the lemma to obtain a map $P_I:\C\to \C$. This is a retraction onto those $\xi$ such that the $\Hop$-support is empty, that is $P_I$ is a retraction onto $I$. As each $P^{h,k}$ is contractive we deduce that $P_I$ is also contractive. We have thus constructed a map $P_I$ with the required properties.

We now move on to the construction of $P_A$. Let $\Hlt=\{(h,k):h,k\in H,\, h<k\}$. We define a map $\plt$ on $[0,1]^2$ by
$$\plt(\xi_1,\xi_2)=\begin{cases}
(\xi_h+\xi_k,0) &\text{if } \xi_h+\xi_k\leq 1\\
(1,\xi_h+\xi_k-1) &\text{if } \xi_h+\xi_k\geq 1
\end{cases}.$$
For $(h,k)$ in $\Hlt$, we define a map $P_h^k$ on $I$ by $P_h^k(\xi)=\xi'$ where $(\xi'_h,\xi'_k)=\plt(\xi_h,\xi_k)$ and $\xi'_j=\xi_j$ for $j\neq h,k$.

Take $(\xi_1,\xi_2),(\eta_1,\eta_2)$ in $[0,1]^2$. If $\xi_1+\xi_2\leq 1$ and $\eta_1+\eta_2\leq 1$ then
$$\normi{\pop(\xi_1,\xi_2)-\pop(\eta_1,\eta_2)}=\normi{(\xi_h+\xi_k)-(\eta_h+\eta_k)|,0}\leq \normi{(\xi_1-\eta_1,\xi_2-\eta_2)}.$$
Similarly if $\xi_1+\xi_2\geq 1$ and $\eta_1+\eta_2\geq 1$. Now suppose that $\xi_1+\xi_2\leq 1$ and $\eta_1+\eta_2\geq 1$. Then 
\begin{align*}
\normi{\pop(\xi_1,\xi_2)-\pop(\eta_1,\eta_2)}&=\normi{((\xi_1+\xi_2)-1,0-(\eta_1+\eta_2-1))}\\
&=(1-\xi_1-\xi_2)+(\eta_1+\eta_2-1)\\
&=\eta_1-\xi_1+\eta_2-\xi_2\leq\normi{(\xi_1-\eta_1,\xi_2-\eta_2)}.
\end{align*}
As $\plt$ is contractive each $P_h^k$ is contractive.

To construct $P_A$ we again apply Lemma \ref{infinite-comp}. In this case we must select the ordering $\prec$ with more care to ensure that hypothesis (2) of the lemma is satisfied. We will say that the \emph{length} of a pair $(h,k)$ is $l$ if $l$ is the largest integer such that we can write $h=h_0<h_1<h_2<\dots<h_l=k$. The length is one if and only if $h$ is a predecessor of $k$. We impose an ordering $\prec$ on $\Hlt$, by first declaring $(h,k)\prec(h',k')$ if the length of $(h,k)$ is greater than the length of $(h',k')$, and choosing any order on pairs of the same length.

The $\Hlt$-support of a point $\xi$ is the set of pairs for which $\xi$ is virtual. If $(h,k)$ is such a pair then $k$ lies in the support of $\xi$ as a function and $h<k$, so the $\Hlt$-support is finite. Now suppose that $(h,k)$ is the $\prec$-least element of the $\Hlt$ support of $\xi$. Then in particular $(h,k)$ has maximal length. Suppose that $\xi$ is actual for a pair $(h',k')$, but applying $P_h^k$ makes it virtual. Then one of $h',k'$ must be one of $h,k$. Moreover, as $P_h^k$ reduces the $k$ coordinate, and increases that $h$ coordinate we have either $h=k'$ or $k=h'$. Consider the first of these cases. We have $h'<k'=h<k$, and $\xi_{h'}<1$ while $\xi_k>0$. It follows that $(h',k)$ is a virtual pair for $\xi$ of greater length than $(h,k)$, contradicting $\prec$-minimality. Similarly in the case $k=h'$ we deduce that $(h,k')$ a virtual pair for $\xi$ of greater length than $(h,k)$. Thus applying $P_h^k$ with $(h,k)$ being $\prec$-minimal cannot increase the $\Hlt$-support. The point $P_h^k$ is $(h,k)$-actual by construction, hence condition (2) of the lemma is satisfied. Thus applying the lemma we obtain a map $P_A:I\to I$ which is a retraction of $I$ onto those $\xi$ in $I$ such that the $\Hlt$-support is empty. In other words $P_A$ is a retraction onto $A$ as required. As each $P_h^k$ is contractive we deduce that $P_A$ is also contractive. Each $P_h^k$ preserves the $l^1$-norm, and for each $\xi$, the point $P_A(\xi)$ is the image of $\xi$ under a finite composition of the maps $P_h^k$. Hence $P_A$ preserves the $l^1$-norm.

To conclude, taking the composition $P=P_A\circ P_I$ we have a contractive retraction of $\C$ onto $A$, which preserves the $l^1$-norm on $I$.
\end{proof}

\begin{cor}
If $X$ is equipped with the $l^1$-path metric, then the embedding $X\to A\subset \C$ is an isometry.
\end{cor}

\begin{proof}
By definition, the $l^1$-path metric on $X$ is the metric obtained by taking the infimum of path-lengths, where these path lengths are measured using the $l^1$-metric on each cube of $X$. Identifying $X$ with its image $A$ in $\C$, the $l^1$-path metric is precisely the path metric induced from the $l^1$-metric on $\C$, hence the statement of the corollary amounts to the assertion that the restriction of the $l^1$-metric to $A$ is a path metric.

Let $d_1$ denote the $l^1$-metric on $A$, and $d_p$ the induced path metric. Given two points $\xi,\eta$ in $A$, the line $(1-t)\xi+t\eta$ is a path from $\xi$ to $\eta$ in $\C$ of length $d_1(\xi,\eta)$. Since the projection $P:\C\to A$ is contractive, $P((1-t)\xi+t\eta)$ gives a path from $\xi$ to $\eta$ of length at most $d_1(\xi,\eta)$, and this path lies in $A$. Thus we have $d_p(\xi,\eta)\leq d_1(\xi,\eta)$. On the other hand, given any metric the induced path metric satisfies the converse inequality, that is $d_p(\xi,\eta)\geq d_1(\xi,\eta)$, hence the two metrics coincide as claimed.
\end{proof}

\section{Asymptotic dimension}
\label{sec:interp}

In this section we make use of the existence of controlled colourings to prove that for a finite dimensional CAT(0) cube complex, the asymptotic dimension is bounded above by the dimension. We begin with the definition of asymptotic dimension.

\begin{defn}[\cite{G}]
Let $X$ be a metric space. The \emph{$R$-degree} of a cover $\U$ of $X$ is the supremum over $x\in X$ of the cardinality of $\{U\in\U: B_R(x)\text{ meets }U\}$. The \emph{asymptotic dimension of $X$} is the least $D$ such that for every $R>0$ there is a cover of $X$ by sets of uniformly bounded diameter and with $R$-degree at most $D+1$.
\end{defn}

The following provides a useful characterisation of asymptotic dimension.

\begin{defn}
A map $\phi$ from a metric space $X$ to a cell complex $Y$ is \emph{uniformly cobounded} if there exists $S>0$ such that $\phi^{-1}(\sigma)$ has diameter at most $S$ for all cells $\sigma$ of $Y$. A map $\phi$ between metric spaces $\phi:X\to Y$ is \emph{cobornologous} if for all $R>0$ there exists $S>0$ such that if $d(\phi(x_1),\phi(x_2))\leq R$ then $d(x_1,x_2)\leq S$.
\end{defn}

\begin{defn}
A \emph{uniform simplicial complex} is a simplicial complex $Y$, equipped with the metric it inherits as a subcomplex of the infinite simplex in $l^2(V)$, where $V$ is the vertex set of $Y$.
\end{defn}

\begin{lemma}[\cite{G,BD}]
\label{Gromov-lemma}
A metric space $X$ has asymptotic dimension $D$ if and only if for every $\varepsilon>0$ there exists an $\varepsilon$-Lipschitz uniformly cobounded map from $X$ into a $D$-dimensional uniform simplicial complex.
\end{lemma}

\begin{rmk}
Let $Y$ be a connected uniform simplicial complex with metric $d_u$, and let $d_p$ be the path metric agreeing with $d_u$ on each simplex. Then $d_p$ is the greatest metric which agrees with $d_u$ on each simplex, so in particular $d_u\leq d_p$. It follows that if $\phi:X\to Y$ is $\varepsilon$-Lipschitz for $d_p$ then it is also $\varepsilon$-Lipschitz for $d_u$. Moreover, since simplices have bounded diameter, if $\phi$ is cobornologous then it is uniformly cobounded. Hence for the `if' direction of Lemma \ref{Gromov-lemma} one is free to use either of these metrics, and require $\phi$ to be either uniformly cobounded or cobornologous.
\end{rmk}

\begin{lemma}
\label{Gromov-cubed}
If for each $\varepsilon>0$ there is an $\varepsilon$-Lipschitz uniformly cobounded (or cobornologous) map from $X$ into a $D$-dimensional CAT(0) cube complex $Y$ then $X$ has has asymptotic dimension at most $D$.
\end{lemma}

\begin{proof}
This follows from Lemma \ref{Gromov-lemma} and the above remark. Given an $\varepsilon$-Lipschitz uniformly cobounded map from $X$ into an $n$-dimensional CAT(0) cube complex $Y$, triangulate each cube of $Y$ to obtain a simplicial complex. As $Y$ is finite dimensional the $l^1$-path metric on $Y$, and the simplicial path metric $d_p$ on its triangulation are bi-Lipschitz equivalent, with constant $\lambda$ depending only on the dimension. Hence we obtain a $\lambda\varepsilon$-Lipschitz map from $X$ into a simplicial complex with the path metric $d_p$. As the map is uniformly cobounded for the cubical structure, it is also uniformly cobounded for the simplicial structure as each simplex lies in a cube.

The hypothesis of being cobornologous is stronger than uniformly coboundedness, so either hypothesis is sufficient.
\end{proof}

Let $X$ be a CAT(0) cube complex and let $H$ denote the set of hyperplanes of $X$. Given $c:H\to \{0,1\}$ a colouring of the hyperplanes, let $\Hhat$ denote the set of 0-coloured hyperplanes. As discussed in section \ref{prelim}, by \cite{CN,BN} there is a CAT(0) cube complex $\Xhat$ whose hyperplanes are $\Hhat$, and a canonical quotient map $\pi:X\to\Xhat$. Recall that a vertex of $X$ provides a choice of orientation for each $h\in H$, namely the set $\h_x$ of halfspaces containing $x$. On the other hand a vertex of $\Xhat$ is defined to be a set of halfspaces, one corresponding to each hyperplane in $\Hhat$, whose total intersection is non-empty. The map $\pi$ is given simply by forgetting the orientations on all 1-coloured hyperplanes. Using the coordinates provided by embedding the complexes into $l^1(H)$ and $l^1(\Hhat)$ respectively, the map takes a point in $X$ with coordinates $\xi=(\xi_h)_{h\in H}$ to $\pi(\xi)=\xi|_{\Hhat}$.

\begin{lemma}
\label{pi-cobounded}
If $c$ is an $l$-controlled colouring of $H$, then the canonical map $\pi$ is cobornologous on the vertices of $X$.
\end{lemma}

\begin{proof}
Let $x,y$ be vertices of $X$, let $m$ be the median of $x,y,x_0$. If $d(x,m)\geq l+1$ then as the colouring is $l$-controlled and a geodesic from $x$ to $m$ is inward, then there must be at least one 0-coloured hyperplane separating $x,m$. Indeed on any inward geodesic at least one in every $l+1$ hyperplanes is 0-coloured, hence there are at least $\lfloor\frac {d(x,m)}{l+1}\rfloor\geq \frac {d(x,m)}{l+1}-1$ hyperplanes which are 0-coloured and separate $x,m$.

Similarly for the hyperplanes separating $y$ and $m$, hence there are at least $\frac {d(x,m)+d(y,m)}{l+1}-2=\frac {d(x,y)}{l+1}-2$ hyperplanes which are 0-coloured and separate $x,y$. It follows that if the distance between $\pi(x)$ and $\pi(y)$ is at most $R$ then $\frac {d(x,y)}{l+1}-2\leq R$ thus $d(x,y)\leq (l+1)(R+2)$. The map is thus cobornologous on the vertex set.
\end{proof}

The problem with the canonical map is that while it is contractive i.e.\ 1-Lipschitz, it will not be $\varepsilon$-Lipschitz for any $\varepsilon<1$. Indeed if $h$ is a 0-coloured hyperplane, and $x,y$ are vertices which are adjacent across $h$ so that $d(x,y)=1$, then $h$ lies in $\Hhat$ hence $\pi(x),\pi(y)$ are also separated by $h$ and $d(\pi(x),\pi(y))=1$. On the other hand if $x,y$ are far enough apart then they will be separated by some 1-coloured hyperplanes, hence $\pi$ will reduce the distance between them. We therefore need to carry out an interpolation to produce a map which is more evenly contractive. We begin by choosing weights on pairs of hyperplanes.

For $h\in \Hhat$ and $k\in H$ we define a weight
$$w(h,k)=\begin{cases}
\frac l{l+1} & \text{if $h=k$};\\
\frac 1{l+1} & \text{if $k$ is 1-coloured, $h<k$ and there is no 1-coloured hyperplane}\\
& \text{separating $h,k$};\\
0 & \text{otherwise.}
\end{cases}$$
Let $\C=\C_\Hhat$ be the infinite cube in $l^1(\Hhat)$. We define a map $\psi:X\to \C$ as follows. For $x$ in $X$ with coordinates $(\xi_k)_k\in H$ we define $\psi(x)=\zeta$ where
$$\zeta_h=\min\left(1,\sum_{k \in H} w(h,k)\xi_k\right).$$

If $\zeta_h=0$ then we must have $\xi_h=0$ since $w(h,h)>0$. On the other hand if $\zeta_h>0$ then either $\xi_h>0$, or there exists $k$ with $h<k$ and $\xi_k>0$. Since $\xi$ is the coordinate vector of a point $x$ in $X$ we know that it is $(h,k)$-actual, so $\xi_h=1$. Hence $\zeta_h=0$ if and only if $\xi_h=0$. We conclude that the support of $\zeta$ is the restriction to $\Hhat$ of the support of $\xi$, in particular it is finite, hence $\zeta$ lies in the cube $\C$.

\begin{lemma}
\label{psi-contraction}
The map $\psi:X\to \C$ is $\frac l{l+1}$-Lipschitz.
\end{lemma}

\begin{proof}
We will show that the map taking a point $x$ with coordinates $\xi_k$ to $\eta_h=\sum\limits_{k \in H} w(h,k)\xi_k$ is $\frac l{l+1}$-Lipschitz, it being clear that truncating coordinates at 1 will only reduce distances.

Two points $x,x'$ with coordinates $\xi_k,\xi'_k$ have images $\eta_h,\eta'_h$ with
$$\eta_h-\eta'_h=\sum_{k \in H} w(h,k)(\xi_k-\xi'_k).$$
Taking the $l^1$-norm we get
$$\sum_{h \in \Hhat}\sum_{k \in H} w(h,k)|\xi_k-\xi'_k|=\sum_{k \in H}|\xi_k-\xi'_k|\left(\sum_{h \in \Hhat} w(h,k)\right).$$
It thus suffices to show that for each $k$ in $H$ we have $\sum_{h \in \Hhat} w(h,k)\leq \frac l{l+1}$.

If $k$ is 0-coloured then $w(h,k)$ is zero except when $h=k$, so in this case the sum is just $w(h,h)=\frac l{l+1}$. If $k$ is 1-coloured then $\sum\limits_{h \in \Hhat} w(h,k)$ is $\frac 1{l+1}$ times the number of 0-coloured hyperplanes $h$ with $h<k$ and such that no 1-coloured hyperplane separates $h,k$. Taking $y$ to be the vertex adjacent to $k$ of minimal distance from $x_0$, there is a 0-coloured inward geodesic crossing all of these hyperplanes $h$. Hence as the colouring is $l$-controlled there are at most $l$ such hyperplanes. Thus we again have $\sum\limits_{h \in \Hhat} w(h,k)\leq \frac l{l+1}$ as required.
\end{proof}

We now combine these results to prove the existence of an $\frac l{l+1}$-Lipschitz cobornologous map from $X$ to $\Xhat$.

\begin{thm}
\label{contraction-thm}
Let $X$ be a CAT(0) cube complex of dimension $D$, with hyperplanes $H$. Let $c$ be an $l$-controlled colouring of $H$. Then there exists a cobornologous, $\frac l{l+1}$-Lipschitz map from $X$ to the cube complex $\Xhat$ with hyperplanes $\Hhat=c^{-1}(\{0\})$.
\end{thm}

\begin{proof}
Let $\psi:X\to \C=\C_\Hhat$ be as above. Identifying $\Xhat$ with the space of actual intervalic coordinates in $\C$, let $P:\C\to \Xhat$ be the projection provided by Theorem \ref{projection-theorem}. We define a map $\phi:X\to \Xhat$ by $\phi(x)=P(\psi(x))$. This map is $\frac l{l+1}$-Lipschitz as $\psi$ is $\frac l{l+1}$-Lipschitz and $P$ is contractive (Lemma \ref{psi-contraction} and Theorem \ref{projection-theorem}). It remains to prove that $\phi$ is cobornologous.

Let $x,y$ be points of $X$. Then there exist vertices $x',y'$ of $X$ with $d(x,x'),d(y,y')\leq D$ where $D$ is the dimension of $X$. As the colouring is $l$-controlled, following maximal 0-coloured inward geodesics from $x',y'$ we reach vertices $x'',y''$ having no 0-coloured inward edges with $d(x',x'')\leq l$ and $d(y',y'')\leq l$. As $\phi$ is contractive we note that $d(\phi(x),\phi(x'')),d(\phi(y),\phi(y''))\leq l+D$, hence if $d(\phi(x),\phi(y))\leq R$ then $d(\phi(x''),\phi(y''))\leq R+2(l+D)$.

We consider the the images of $x'',y''$ under the map $\psi$. Let $\xi''_k$ be the coordinates of $x''$ and consider $\eta_h=\sum_{k\in H} w(h,k)\xi''_k$. We have already remarked that the support of $\eta$ is the support of the restriction of $\xi''$ to $\Hhat$, in other words it is the set of 0-coloured hyperplanes separating $x_0,x''$. Moreover as $x''$ has no 0-coloured inward edges, for any such $h$ there exists a 1-coloured hyperplane $k$ with $h<k$. It follows that $\eta_h\geq w(h,h)+w(h,k)=1$. Hence $\zeta=\psi(x)$ is precisely the characteristic function of the set of $h$ in $\Hhat$ separating $x_0,x''$. This is the image of $x''$ under the canonical quotient map, that is $\psi(x'')=\pi(x'')$.

Since $\pi(x'')$ lies in $\Xhat$ its coordinates are intervalic and actual so we have
$$\phi(x'')=P(\psi(x''))=P(\pi(x''))=\pi(x'').$$
Similarly $\phi(y'')=\pi(y'')$. We know from Lemma \ref{pi-cobounded} that $\pi$ is uniformly cobounded. Hence given $R>0$ there exists $S$ such that for $s(\phi(x''),\phi(y''))=d(\pi(x''),\pi(y''))\leq R+2(l+D)$ we have $d(x'',y'')\leq S$. It follows that if $d(\phi(x),\phi(y))\leq R$ then $d(x,y)\leq S+2(l+D)$, i.e.\ $\phi$ is cobornologous.
\end{proof}

Applying this we deduce that finite dimensional CAT(0) cube complexes have finite asymptotic dimension.

\begin{thm}
\label{main-thm}
Let $X$ be a CAT(0) cube complex of dimension $D$. Then for all $\varepsilon>0$ there exists an $\varepsilon$-Lipschitz cobornologous map from $X$ to a CAT(0) cube complex of dimension at most $D$. Thus $X$ has asymptotic dimension at most $D$.
\end{thm}

\begin{proof}
By Theorem \ref{colouring} a $D$-dimensional CAT(0) cube complex $X$ admits a $3^{D-1}D$-controlled colouring, as the flatness of $X$ is at most $D$. Thus by Theorem \ref{contraction-thm} there exists a $\frac {3^{D-1}D}{3^{D-1}D+1}$ contractive cobornologous map from $X$ into $\Xhat$, where $\Xhat$ is again a CAT(0) cube complex of dimension at most $D$. Iterating the process we can produce an $\varepsilon$-Lipschitz map for any $\varepsilon>0$, and as this is a composition of cobornologous maps it is again cobornologous as required. The asymptotic dimension bound now follows from Lemma \ref{Gromov-cubed}.
\end{proof}

If $G$ acts properly isometrically on a CAT(0) cube complex (or indeed any metric space) then each orbit is coarsely equivalent to the group $G$. Since asymptotic dimension is monotonic, one obtains the following immediate consequence.

\begin{cor}
\label{main-cor1}
Let $G$ be a group admitting a proper isometric action on a CAT(0) cube complex of dimension $D$. Then $G$ has asymptotic dimension at most $D$.
\end{cor}

As an application of this result, we examine the asymptotic dimension bound this gives us for small cancellation groups.

\begin{ex}
\label{small-cancellation}
In \cite{W}, Wise showed that if $G$ is a $B(4)-T(4)$ or $B(6)$ small cancellation group then $G$ acts properly isometrically on a CAT(0) cube complex built from the presentation 2-complex. The prsentation complex should be subdivided if necessary to ensure cells have even circumference. In the $B(4)-T(4)$ case the maximal cubes correspond to cells, and their dimension is half the circumference of the cell. Hence the CAT(0) cube complex has dimension $c/2$ where $c$ is the maximal circumference of a cell in the presentation complex of $G$. By Corollary \ref{main-cor1}, it follows that for a $B(4)-T(4)$ group, the asymptotic dimension is at most $c/2$.

In the $B(6)$ case the maximal cubes correspond to cells, links or tricombs \cite[Lemma 9.4]{W}. Cubes corresponding to cells again have dimension at most $c/2$. For each tricomb cube, one can find two cells in the presentation complex (a corner and a cell on the opposite side of the tricomb), such that all hyperplanes cross one or other of these cells. It follows that there are at most $2c/2=c$ hyperplanes, so these cubes have dimension at most $c$. Finally, cubes corresponding to links have dimension at most $l$ where $l$ is the maximal size of a complete graph in the generalised link of a vertex. For $B(6)$ small cancellation groups we therefore deduce that the asymptotic dimension is at most $\max\{c,l\}$.
\end{ex}

\end{document}